\documentclass[12pt,reqno]{amsart}
\usepackage{fullpage}
\usepackage{amsmath,amssymb,amsthm}
\usepackage{mathtools}
\usepackage{booktabs}
\usepackage{tikz}
\usepackage{colonequals}
\usetikzlibrary{arrows.meta,calc,positioning}
\usepackage[colorlinks=true,linkcolor=blue,citecolor=blue,urlcolor=blue]{hyperref}

\newtheorem{thm}{Theorem}[section]
\newtheorem{lem}[thm]{Lemma}

\newtheorem{cor}[thm]{Corollary}

\theoremstyle{definition}
\newtheorem{defn}[thm]{Definition}
\newtheorem{rmk}[thm]{Remark}
\newtheorem{ex}[thm]{Example}
\newtheorem{quest}[thm]{Question}

\title{Defective chromatic polynomials}
\author{Shamil Asgarli}
\address{Department of Mathematics \& Computer Science \\ Santa Clara University \\ CA 95053} \email{sasgarli@scu.edu}

\author{Tamsen Whitehead McGinley}
\address{Department of Mathematics \& Computer Science \\ Santa Clara University \\ CA 95053} \email{tmcginley@scu.edu}

\author{Nicholas Xue}
\address{Department of Mathematics \& Computer Science \\ Santa Clara University \\ CA 95053} \email{nxue@scu.edu}

\subjclass[2020]{Primary: 05C15, 05C31; Secondary: 05C05, 05C60}
\keywords{defective coloring, defective chromatic polynomial, degree chromatic polynomial, trees, matching polynomial, cospectral graph, counting linear forests}

\begin{document}

\begin{abstract}
For a graph $G$ and an integer $d\geq 0$, the defective chromatic polynomial $\chi_d(G;k)$ counts the $k$-colorings of $G$ in which each vertex has at most $d$ neighbors of its own color. We investigate which structural properties of $G$ are determined by the full family $\{\chi_d(G;k)\}_{d\geq 0}$. We establish a contraction formula expressing $\chi_d(G;k)$ as a sum of ordinary chromatic polynomials of the edge contractions of $G$. As a first application, we prove that for triangle-free graphs, the full family determines the degree sequence. For trees, we show further that the family $\{\chi_d(T;k)\}_{d\geq 0}$ determines the path-subgraph counts $N(P_j,T)$ for $j=1,2,3,4$, but \emph{not} for $j=5$. For each $n\geq 9$, we construct a pair of nonisomorphic trees of order $n$ that share the same defective chromatic polynomials for every $d\geq 0$.
\end{abstract}

\maketitle

\section{Introduction}

Let $G$ be a finite simple graph. A \emph{$(k,d)$-coloring} of $G$ is a map
\[
c\colon V(G)\to [k]\colonequals \{1,2,\dots,k\}
\]
such that every vertex $v\in V(G)$ has at most $d$ neighbors $w\in N_G(v)$ with $c(w)=c(v)$. Equivalently, the color classes induce subgraphs of maximum degree at most~$d$.

\begin{defn}
For a graph $G$ and an integer $d\geq 0$, we write $\chi_d(G;k)$ for the number of $(k,d)$-colorings of $G$ and call it the \emph{defective chromatic polynomial} of $G$.
\end{defn}

Two extreme cases are immediate: setting $d=0$ recovers the ordinary chromatic polynomial $\chi_0(G;k)=\chi(G;k)$, while setting $d\geq \Delta(G)$ imposes no restriction, yielding $\chi_d(G;k)=k^{|V(G)|}$. The intermediate range, where the new phenomena arise, is $1\leq d\leq \Delta(G)-1$.

The name \emph{degree chromatic polynomial} also appears in the literature, under a slightly different indexing convention. This terminology is due to Humpert and Martin \cite{HM12}, who introduced the polynomial as part of a Hopf-algebraic framework. A conjecture of Humpert and Martin on the leading terms for trees was later proved by Cifuentes \cite{Cif11, Cif12}. In their notation, $P_m(G,k)$ counts the $k$-colorings in which every vertex has fewer than $m$ neighbors of its own color. The two definitions are related by $\chi_d(G;k)=P_{d+1}(G,k)$. Either way, the abbreviation DCP is consistent with both names.

Defective colorings belong to the broader theory of improper colorings and $(m,k)$-colorings; see Cowen, Cowen, and Woodall \cite{CCW86}, Frick \cite{Fri93}, and Cowen, Goddard, and Jesurum \cite{CGJ97}. The defective chromatic polynomial is a special case of the Harary polynomial of Herscovici, Makowsky, and Rakita \cite{HMR21}. For background on chromatic polynomials, see Read \cite{Rea68}.

We also observe that the problem admits a formulation in terms of hypergraphs. For a graph $G$ and an integer $d\geq 0$, let $H_d(G)$ be the hypergraph with vertex set $V(G)$ whose hyperedges are the subsets $W\subseteq V(G)$ for which the induced subgraph $G[W]$ has maximum degree at least $d+1$. Then $\chi_d(G;k)$ is the chromatic polynomial of $H_d(G)$; recall that a proper hypergraph coloring is a vertex coloring in which no hyperedge is monochromatic. This construction fits the general framework of Brown and Corneil \cite{BC91} for translating graph colorings into hypergraph colorings when the condition imposed on each color class is preserved under taking induced subgraphs. For background on hypergraphs and their chromatic polynomials, see \cite{Ber73,ZD20}. We have not explored whether this viewpoint yields new results on defective chromatic polynomials, so we simply mention the connection here for interested readers.

The defective chromatic polynomial should be distinguished from the \emph{$q$-defect polynomial}, which counts colorings with exactly $q$ monochromatic edges. That invariant was studied by Mphako-Banda \cite{M-B19}; for a broader discussion of related graph polynomials, see Ellis-Monaghan and Merino \cite{E-MM11}. Both invariants measure how far a coloring is from being proper: the $q$-defect polynomial counts the total number of monochromatic edges, while $\chi_d(G;k)$ bounds the monochromatic degree at each vertex.

The ordinary chromatic polynomial is well known to be an incomplete invariant of graphs, and even of trees. Allowing a defect parameter produces a family of invariants, and it is natural to ask what additional information this family carries:

\begin{quest}
Which properties of a graph $G$ are determined by the full family
\[
\{\chi_d(G;k): d\geq 0\}
\]
of its defective chromatic polynomials?
\end{quest}

Our first result is that for triangle-free graphs, the full defective family determines the degree sequence.

\begin{thm}\label{thm:intro-degree-sequence}
Let $G$ be a triangle-free graph. Then the full family $\{\chi_d(G;k): d\geq 0\}$ determines the degree sequence of $G$.
\end{thm}

For trees, we prove more: the defective polynomials determine the counts of short paths. Let $N(H, G)$ denote the number of subgraphs of $G$ isomorphic to $H$.

\begin{thm}\label{thm:intro-path-counts}
Let $T_1$ and $T_2$ be two trees with $\chi_d(T_1;k)=\chi_d(T_2;k)$ for every $d\geq 0$. Then
\[
N(P_j,T_1)=N(P_j,T_2)\qquad\text{for each }j\in\{1,2,3,4\}.
\]
\end{thm}

On the other hand, we exhibit such trees for which $N(P_5,T_1)\neq N(P_5,T_2)$, showing that Theorem~\ref{thm:intro-path-counts} is sharp. In particular, the full defective family is not a complete invariant, even for trees; in fact, infinitely many pairs of trees witness this failure.

\begin{thm}\label{thm:intro-infinite}
There exist infinitely many pairs of nonisomorphic trees having the same full family of defective chromatic polynomials.
\end{thm}

In the final section, we show that the DCP family is incomparable with both the two-variable polynomial of Dohmen, P\"onitz, and Tittmann \cite{DPT03} and the generalized degree polynomial of Crew \cite{Cre20}.

\medskip
\noindent\textbf{Organization.}
Section~\ref{sec:contraction} establishes a formula for $\chi_d(G;k)$ as a sum of ordinary chromatic polynomials of contracted graphs. Theorems~\ref{thm:intro-degree-sequence} and~\ref{thm:intro-path-counts} are proved in Sections~\ref{sec:degree-sequence} and~\ref{sec:paths}, respectively. In Section~\ref{sec:infinite}, we prove Theorem~\ref{thm:intro-infinite} by constructing an explicit infinite family of trees. Finally, Section~\ref{sec:comparison} compares DCP with several other graph polynomials on trees.

\section{A contraction formula}\label{sec:contraction}

Every coloring of $G$ has a naturally associated set of \emph{monochromatic edges}. If the spanning subgraph with this edge set has maximum degree at most $d$, the coloring descends to a proper coloring of the graph obtained by contracting these edges. Summing over all such edge sets yields a formula for $\chi_d(G;k)$ in terms of ordinary chromatic polynomials.

For $A\subseteq E(G)$, write $G[A]$ for the spanning subgraph of $G$ with edge set $A$, and $G/A$ for the graph obtained by simultaneously contracting every edge in $A$. We regard $G/A$ as a multigraph, allowing loops and parallel edges. The chromatic polynomial is defined as usual; in particular, a loop forces it to vanish, while parallel edges have no effect on proper colorings.

\begin{thm}\label{thm:general-contraction}
Let $G$ be a graph and let $d\geq 0$. Then
\[
\chi_d(G;k)=\sum_{\substack{A\subseteq E(G)\\ \Delta(G[A])\leq d}} \chi_0(G/A;k).
\]
\end{thm}

\begin{proof}
Given a $(k,d)$-coloring $c$ of $G$, let
\[
A_c\colonequals \{uv\in E(G): c(u)=c(v)\}
\]
be the set of monochromatic edges. Since the number of monochromatic neighbors of a vertex $v$ equals its degree in $G[A_c]$, the defect condition on $c$ is equivalent to $\Delta(G[A_c])\leq d$. Contracting $A_c$, we see that $c$ descends to a coloring of $G/A_c$, since every edge of $A_c$ is monochromatic. No loop arises in $G/A_c$: if an edge $uv\notin A_c$ had both endpoints in the same connected component of $G[A_c]$, then $u$ and $v$ would share a color, forcing $uv\in A_c$, a contradiction. Hence the induced coloring on $G/A_c$ is proper.

Conversely, given $A\subseteq E(G)$ with $\Delta(G[A])\leq d$, each proper coloring of $G/A$ lifts uniquely to a $(k,d)$-coloring of $G$ with $A_c=A$. The lift assigns each vertex $v\in V(G)$ the color of the vertex of $G/A$ containing~$v$. Every edge of~$A$ is then monochromatic by construction. If $uv\notin A$ and $u,v$ were identified by the contraction, $G/A$ would contain a loop, contradicting the existence of a proper coloring; otherwise $u$ and $v$ correspond to distinct adjacent vertices of $G/A$ and receive different colors. Hence no edge outside $A$ is monochromatic.

Thus, the $(k, d)$-colorings of $G$ are partitioned by their monochromatic edge set $A=A_{c}$, and for each admissible $A$, the contribution is $\chi_0(G/A;k)$.
\end{proof}

In particular, since each $\chi_0(G/A;k)$ is a polynomial in $k$, the formula shows that $\chi_d(G;k)$ is a polynomial in $k$.

As a consistency check, consider the two extreme values of $d$. For $d=0$, the only admissible set is $A=\emptyset$, so the formula reduces to $\chi_0(G;k)$. For $d\geq \Delta(G)$, every $A\subseteq E(G)$ is admissible, and the formula partitions all $k^{|V(G)|}$ colorings by their monochromatic edge set.

When $G$ is a tree, contracting any edge set yields a smaller tree, and the sum simplifies.

\begin{cor}\label{cor:tree-expansion}
Let $T$ be a tree on $n$ vertices and let $d\geq 0$. For each $r\in \{0,1,\dots,n-1\}$, let
\[
c_{r,\leq d}(T)\colonequals \#\{A\subseteq E(T) : |A|=r,\ \Delta(T[A])\leq d\}.
\]
Then
\[
\chi_d(T;k)=\sum_{r=0}^{n-1} c_{r,\leq d}(T)\,k(k-1)^{n-r-1}.
\]
\end{cor}

\begin{proof}
Every subset $A\subseteq E(T)$ gives a forest $T[A]$. Contracting $r=|A|$ of these edges produces a tree on $n-r$ vertices with chromatic polynomial $k(k-1)^{n-r-1}$. Grouping the terms in Theorem~\ref{thm:general-contraction} by $|A|$ gives the result.
\end{proof}

Since the polynomials $\{k(k-1)^{n-r-1}\}_{r=0}^{n-1}$ have distinct degrees, they are linearly independent, and so the coefficients $c_{r,\leq d}(T)$ are uniquely determined by $\chi_d(T;k)$.

\begin{rmk}\label{rmk:extraction}
For a tree $T$, define 
\[
c_{r,d}(T)\colonequals \#\{A\subseteq E(T) : |A|=r,\ \Delta(T[A])=d\}.
\]
Then
\[
c_{r,\leq d}(T)=\sum_{i=0}^d c_{r,i}(T)
\qquad\text{and}\qquad
c_{r,d}(T)=c_{r,\leq d}(T)-c_{r,\leq d-1}(T),
\]
where we adopt the convention $c_{r,\leq -1}(T)=0$. Hence $\{\chi_d(T;k)\}_{d\geq 0}$ determines both arrays $\{c_{r,\leq d}(T)\}$ and $\{c_{r,d}(T)\}$.
\end{rmk}

The coefficients $c_{r,\leq d}(T)$ also have familiar combinatorial meanings for small $d$: $c_{r,\leq 1}(T)$ counts the $r$-edge matchings of $T$, while $c_{r,\leq 2}(T)$ counts the $r$-edge linear forests of $T$. Both interpretations will play a central role in Section~\ref{sec:infinite}.

\section{Degree sequence reconstruction for triangle-free graphs}\label{sec:degree-sequence}

We now apply the contraction formula to recover the degree sequence of a triangle-free graph from its full defective family. The proof applies inclusion--exclusion to the colorings that violate the defect condition, combined with a degree bound (Lemma~\ref{lem:degree-bound}) and a bound on pairwise intersections (Lemma~\ref{lem:pairwise-intersections}). Only the second lemma uses the triangle-free hypothesis.

For a graph $G$, let $M_r(G)$ denote the number of vertices of degree $r$. If $A$ is a finite set and $t\geq 0$, then $\binom{A}{t}$ denotes the collection of all subsets of $A$ of size $t$.

\begin{lem}\label{lem:degree-bound}
Let $G$ be a graph on $n$ vertices and let $d\geq 0$. Then the polynomial $k^n-\chi_d(G;k)$ has degree at most $n-d-1$; equivalently,
\[
k^n-\chi_d(G;k)=O(k^{n-d-1}).
\]
\end{lem}

\begin{proof}
Applying Theorem~\ref{thm:general-contraction} with $d=\Delta(G)$ gives
\[
k^n=\sum_{A\subseteq E(G)} \chi_0(G/A;k),
\]
since every edge set is admissible in this case. Subtracting the terms satisfying $\Delta(G[A])\leq d$ yields
\[
k^n-\chi_d(G;k)
=
\sum_{\substack{A\subseteq E(G)\\ \Delta(G[A])>d}} \chi_0(G/A;k).
\]
For each such $A$, the graph $G[A]$ has a vertex of degree at least $d+1$, so some connected component of $G[A]$ has at least $d+2$ vertices. Contracting that component lowers the vertex count by at least $d+1$, and hence
\[
\deg \chi_0(G/A;k)\leq n-d-1.
\]
Since the number of terms is independent of $k$, the sum itself has degree at most $n-d-1$.
\end{proof}

For $v\in V(G)$ and $U\in \binom{N_G(v)}{d}$, set
\[
Q(v,U)\colonequals \{v\}\cup U
\qquad\text{and}\qquad
R(v,U)\colonequals \{c\colon V(G)\to [k] : c \text{ is constant on } Q(v,U)\}.
\]
A coloring fails to be $(d-1)$-defective if and only if it lies in $R(v,U)$ for some such pair. The key step in the inclusion--exclusion argument below is bounding the pairwise intersections of these events. The next lemma shows that, in the triangle-free setting, all such intersections are already of lower order.

\begin{lem}\label{lem:pairwise-intersections}
Let $G$ be a triangle-free graph on $n$ vertices, and let $d\geq 2$. If $(v,U)\neq (w,W)$ are two pairs with $v,w\in V(G)$, $U\in \binom{N_G(v)}{d}$, and $W\in \binom{N_G(w)}{d}$, then
\[
|R(v,U)\cap R(w,W)|\leq k^{n-d-1}.
\]
\end{lem}

\begin{proof}
Write $Q_1=Q(v,U)$ and $Q_2=Q(w,W)$; both have size $d+1$. A coloring in $R(v,U)\cap R(w,W)$ is constant on $Q_1$ and on $Q_2$.

If $Q_1\cap Q_2=\emptyset$, then
\[
|R(v,U)\cap R(w,W)|=k^{n-2d}\leq k^{n-d-1},
\]
since $d\geq 1$.

Otherwise, every coloring in the intersection is constant on $Q_1\cup Q_2$, so
\[
|R(v,U)\cap R(w,W)|\leq k^{n-|Q_1\cup Q_2|+1}.
\]
Since $|Q_1|=|Q_2|=d+1$, the only way to have $|Q_1\cup Q_2|<d+2$ is $Q_1=Q_2$; we show this is impossible.

If $v=w$, then $Q_1=Q_2$ forces $U=W$, contradicting $(v,U)\neq(w,W)$. If $v\neq w$, then $Q_1=Q_2$ forces $w\in U$ and $v\in W$, so $vw\in E(G)$. Since $d\geq 2$, there exists $x\in U\setminus\{w\}$; because $Q_1=Q_2$ we also have $x\in W$, so $x$ is adjacent to both $v$ and $w$, yielding a triangle, a contradiction. Hence $|Q_1\cup Q_2|\geq d+2$, so
\[
|R(v,U)\cap R(w,W)|\leq k^{n-|Q_1\cup Q_2|+1}\leq k^{n-(d+2)+1}=k^{n-d-1},
\]
as desired.
\end{proof}

We now combine these two lemmas to prove Theorem~\ref{thm:intro-degree-sequence}.

\begin{proof}[Proof of Theorem~\ref{thm:intro-degree-sequence}]
Fix $d\geq 2$. A coloring fails to be $(d-1)$-defective if and only if some vertex $v$ has at least $d$ neighbors of its own color, that is, if and only if the coloring lies in $R(v,U)$ for some $v\in V(G)$ and some $U\in \binom{N_G(v)}{d}$. Hence
\begin{equation}\label{eq:big-union}
k^n-\chi_{d-1}(G;k)=\left|\bigcup_{\substack{v\in V(G)\\ U\in \binom{N_G(v)}{d}}} R(v,U)\right|.
\end{equation}
The first sum of inclusion--exclusion, applied to the right-hand side of \eqref{eq:big-union}, with $|R(v,U)|=k^{n-d}$, gives
\[
\sum_{v\in V(G)} \binom{\deg(v)}{d}k^{n-d}
=
\left(\sum_{r=d}^{\Delta} M_r(G)\binom{r}{d}\right)k^{n-d}.
\]
By Lemma~\ref{lem:pairwise-intersections}, every pairwise intersection is $O(k^{n-d-1})$. Every higher-order intersection is contained in a pairwise intersection, and the number of terms in the inclusion--exclusion expansion is independent of $k$. Hence
\[
k^n-\chi_{d-1}(G;k)
=
\left(\sum_{r=d}^{\Delta} M_r(G)\binom{r}{d}\right)k^{n-d}
+O(k^{n-d-1}).
\]
On the other hand, Lemma~\ref{lem:degree-bound} gives $k^n-\chi_d(G;k)=O(k^{n-d-1})$. Subtracting,
\[
\chi_d(G;k)-\chi_{d-1}(G;k)
=
\left(\sum_{r=d}^{\Delta} M_r(G)\binom{r}{d}\right)k^{n-d}
+O(k^{n-d-1}).
\]
The maximum degree $\Delta=\Delta(G)$ is determined as the least integer $d\geq 0$ for which $\chi_d(G;k)=k^n$. Assume $\Delta\geq 2$; the case $\Delta\leq 1$ will be handled at the end.

We now recover the degree sequence. For $d=\Delta$, the leading coefficient is $M_\Delta(G)$. For $d=\Delta-1$ the leading coefficient is
\[
M_\Delta(G)\binom{\Delta}{\Delta-1}+M_{\Delta-1}(G),
\]
so $M_{\Delta-1}(G)$ is determined. Iterating this argument down to $d=2$, we determine $M_\Delta(G)$, $M_{\Delta-1}(G)$, $\dots$, $M_2(G)$ in turn. 

It remains to find $M_1(G)$ and $M_0(G)$. The vertex count $n$ is the degree of $\chi_0(G;k)$. Let $m$ denote the number of edges. Since the coefficient of $k^{n-1}$ in $\chi_0(G;k)$ is $-m$, we also obtain the sum of all degrees, which is $2m$. We have
\begin{equation}\label{eq:M_0-M_1}
\sum_{r\geq 0} M_r(G)=n
\qquad\text{and}\qquad
\sum_{r\geq 0} r\,M_r(G)=2m.
\end{equation}
Since $M_r(G)$ is already known for all $r\geq 2$, equations \eqref{eq:M_0-M_1} determine $M_1(G)$ and $M_0(G)$. Finally, if $\Delta\leq 1$, then we only need $M_0(G)$ and $M_1(G)$, which are determined by \eqref{eq:M_0-M_1}.
\end{proof}

Since trees are triangle-free, we immediately obtain the following result.

\begin{cor}\label{cor:trees-degree-sequence}
If two trees have the same family of defective chromatic polynomials, then they have the same degree sequence.
\end{cor}

\section{Path counts for trees}\label{sec:paths}

We now ask what further structural information the defective family records beyond the degree sequence. A natural candidate is the number $N(P_j,T)$ of copies of the path $P_j$ in $T$, which is the focus of this section.

By Corollary~\ref{cor:tree-expansion} and Remark~\ref{rmk:extraction}, the defective family $\{\chi_d(T;k)\}_{d\geq 0}$ determines the arrays $\{c_{r,\leq d}(T)\}$ and $\{c_{r,d}(T)\}$. To prove Theorem~\ref{thm:intro-path-counts}, it therefore suffices to express $N(P_j,T)$ for $j\in\{1,2,3,4\}$ in terms of the degree sequence and the invariants $c_{r,d}(T)$.

We separate the easy cases $j\leq 3$, which follow directly from the degree sequence, from the more involved case $j=4$, which requires a double count involving $P_3\sqcup P_2$. Throughout this section, we use the shorthand $V=V(T)$ and $E=E(T)$, and we write $d_v=\deg(v)$ for the degree of a vertex $v$.

\begin{lem}\label{lem:easy-path-counts}
For a tree $T$ on $n$ vertices,
\[
N(P_1, T) = n, \qquad N(P_2, T)= n-1, \qquad N(P_3, T) = c_{2,2}(T) = \sum_{v\in V} \binom{d_v}{2}.
\]
In particular, $2c_{2,2}(T)=\bigl(\sum_{v\in V} d_v^2\bigr)-2(n-1)$.
\end{lem}

\begin{proof}
The first two equalities are immediate, because $N(P_1,T)=|V|=n$ and $N(P_2,T)=|E|=n-1$. For the third, a subset $A\subseteq E$ with $|A|=2$ and $\Delta(T[A])=2$ is precisely the edge set of a copy of $P_3$ in $T$, so $N(P_3, T)=c_{2,2}(T)$. The identity $N(P_3,T)=\sum_{v\in V}\binom{d_v}{2}$ follows from classifying each $P_3$ according to its central vertex: for each $v\in V$, the copies of $P_3$ centered at $v$ correspond bijectively to the $\binom{d_v}{2}$ unordered pairs of neighbors of $v$. Equating the two expressions for $N(P_3, T)$, rearranging, and using $\sum_{v\in V}d_v=2(n-1)$, we obtain
\[
2c_{2,2}(T)=\sum_{v\in V}(d_v^2-d_v)=\left(\sum_{v\in V}d_v^2\right)-2(n-1). \qedhere
\]
\end{proof}

Before turning to $N(P_4,T)$, recall that the \emph{first} and \emph{second Zagreb indices} of a graph $G$ are defined, respectively, by
\begin{align*}
    Z_1(G) = \sum_{uv\in E(G)} (d_u+d_v) \qquad \text{and} \qquad Z_2(G) =  \sum_{uv\in E(G)} d_u d_v.
\end{align*}
The first Zagreb index can also be expressed as $\sum_{v\in V(G)} d_v^2$. Indeed, in the sum $\sum_{uv\in E(G)}(d_u+d_v)$, each vertex $u$ contributes the term $d_u$ once for each of its incident edges, for a total of $d_u^2$. In particular, $Z_1$ is determined by the degree sequence. The second Zagreb index $Z_2$ is not, but Corollary~\ref{cor:zagreb-dcp} below shows that it is nonetheless determined by the DCP family.

\begin{lem}\label{lem:P4-count}
For a tree $T$ on $n$ vertices,
\[
N(P_4, T) = \sum_{x\in V}\binom{d_x}{2}(n+1-d_x) - c_{3,2}(T) - 2c_{2,2}(T).
\]
\end{lem}

\begin{proof}
Each copy of $P_4$ is determined by its middle edge $uv\in E$. To extend $uv$ to a $P_4$, we select one of the $d_u-1$ other neighbors of $u$ and one of the $d_v-1$ other neighbors of $v$. Hence
\begin{align*}
    N(P_4, T) &= \sum_{uv\in E} (d_u-1)(d_v-1) \\
    &= \sum_{uv\in E} (d_u d_v - d_u - d_v + 1) \\
    &=\sum_{uv\in E}d_ud_v -\sum_{uv\in E}(d_u+d_v)+\sum_{uv\in E}1\\
    &=Z_2(T) - \sum_{v\in V} d_v^2 + (n-1).
\end{align*}

The expression above involves $Z_2(T)$, which is not determined by the degree sequence. We will eliminate $Z_2(T)$ by counting copies of $P_3\sqcup P_2$ (the disjoint union of $P_3$ and $P_2$) and using the identity $c_{3, 2}(T)=N(P_3\sqcup P_2, T)+N(P_4,T)$.

Consider the set
\[
I = \{(Q, R) \mid Q, R \text{ are subgraphs of } T, \ Q\cong P_3,\ R\cong P_2,\ \text{ and } V(Q)\cap V(R)=\emptyset\}.
\]
In any copy of $P_3\sqcup P_2$ in the tree, there is a unique choice for the subgraphs $Q\cong P_3$ and $R\cong P_2$ with $V(Q)\cap V(R)=\emptyset$. Thus $|I|=N(P_3\sqcup P_2, T)$.

We now compute $|I|$ in a different way. Fix a copy $Q_0\cong P_3$, say $Q_0$ is given by $y\text{--}x\text{--}z$. The number of edges $R\cong P_2$ such that $V(Q_0)\cap V(R)=\emptyset$ is precisely the number of edges in $T\setminus \{x,y,z\}$. When we remove the vertices $x,y,z$ from $T$, we delete exactly $d_x+d_y+d_z-2$ edges. Consequently,
\[
\text{number of edges in } T\setminus \{x,y,z\}
=(n-1)-(d_x+d_y+d_z-2)=n-(d_x+d_y+d_z)+1.
\]
Summing over all choices of $Q_0$, we obtain
\begin{align*}
|I| &= \sum_{y\text{--}x\text{--}z} (n-(d_x+d_y+d_z)+1) \\
&= \sum_{x\in V} \sum_{\{y,z\}\subseteq N(x)} (n+1-d_x-(d_y+d_z)) \\
&=\sum_{x\in V}\sum_{\{y,z\}\subseteq N(x)}(n+1-d_x) - \sum_{x\in V}\sum_{\{y,z\}\subseteq N(x)} (d_y+d_z) \\
&=\sum_{x\in V} \binom{d_x}{2}(n+1-d_x) - \sum_{x\in V} \sum_{y\in N(x)} (d_x-1)d_y.
\end{align*}
The second sum simplifies when rewritten as a sum over edges. Each edge $uv \in E$ contributes two terms: $(d_u-1)d_v$ and $(d_v-1)d_u$. Thus,
\begin{align*}
|I| &= \sum_{x\in V}\binom{d_x}{2}(n+1-d_x) - \sum_{uv\in E}((d_u-1)d_v+(d_v-1)d_u) \\
&= \sum_{x\in V}\binom{d_x}{2}(n+1-d_x)-2\sum_{uv\in E} d_u d_v + \sum_{uv\in E}(d_u+d_v) \\
&= \sum_{x\in V}\binom{d_x}{2}(n+1-d_x) - 2 Z_2(T) + \sum_{v\in V} d_v^2.
\end{align*}
Since $|I|=N(P_3\sqcup P_2, T)$, we have
\begin{equation}\label{eq:P3uP2}
N(P_3\sqcup P_2, T) = \sum_{x\in V}\binom{d_x}{2}(n+1-d_x) - 2 Z_2(T) + \sum_{v\in V} d_v^2.
\end{equation}
A subset $A\subseteq E(T)$ with $|A|=3$ and $\Delta(T[A])=2$ is the edge set of a copy of $P_3\sqcup P_2$ or $P_4$, so
\[
c_{3,2}(T) = N(P_3\sqcup P_2, T) + N(P_4,T).
\]
Substituting \eqref{eq:P3uP2} and
\begin{equation}\label{eq:P4-Z2}
N(P_4, T) = Z_2(T) -\sum_{v\in V} d_v^2 + (n-1)
\end{equation}
gives
\[
c_{3,2}(T) = \sum_{x\in V} \binom{d_x}{2}(n+1-d_x) -Z_2(T) + (n-1),
\]
so
\begin{equation}\label{eq:Z2-formula}
Z_2(T) = \sum_{x\in V}\binom{d_x}{2}(n+1-d_x)+(n-1)-c_{3,2}(T).
\end{equation}
Substituting \eqref{eq:Z2-formula} into \eqref{eq:P4-Z2} yields
\[
N(P_4, T) =  \sum_{x\in V} \binom{d_x}{2}(n+1-d_x) - c_{3,2}(T) + 2(n-1) - \sum_{v\in V} d_v^2.
\]
Finally, by Lemma~\ref{lem:easy-path-counts}, $\sum_{v}d_v^2-2(n-1)=2c_{2,2}(T)$, which gives the cleaner formula
\[
N(P_4, T) = \sum_{x\in V}\binom{d_x}{2}(n+1-d_x) - c_{3,2}(T) - 2c_{2,2}(T). \qedhere
\]
\end{proof}

\begin{proof}[Proof of Theorem~\ref{thm:intro-path-counts}]
By Corollary~\ref{cor:trees-degree-sequence} and Remark~\ref{rmk:extraction}, if $T_1$ and $T_2$ have the same defective chromatic polynomials, they share the same degree sequence and satisfy $c_{r,d}(T_1)=c_{r,d}(T_2)$ for all $r,d$. Each of the formulas
\[
N(P_1, T) = n, \qquad N(P_2, T)=n-1, \qquad N(P_3, T)=\sum_{v\in V}\binom{d_v}{2},
\]
\[
N(P_4, T) = \sum_{x\in V}\binom{d_x}{2}(n+1-d_x) - c_{3,2}(T) - 2c_{2,2}(T)
\]
from Lemmas~\ref{lem:easy-path-counts} and~\ref{lem:P4-count} depends only on these invariants. Hence $N(P_j,T_1)=N(P_j,T_2)$ for $j\in\{1,2,3,4\}$.
\end{proof}

We close with a byproduct of Lemma~\ref{lem:P4-count}. Unlike $Z_1$, the second Zagreb index is not determined by the degree sequence alone: the two caterpillars below share the degree sequence $(3,2,2,1,1,1)$ but have different values of $Z_2$, namely $18$ for the left tree and $19$ for the right.

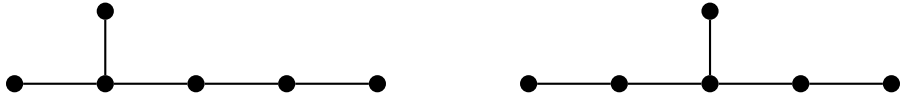
\begin{center}
\begin{tikzpicture}[
    scale=0.73,
    thick,
    every node/.style={circle,draw,fill=black,inner sep=2pt}
]

    \begin{scope}[xshift=0cm]
        \node (a1) at (0, 0)   {};
        \node (a2) at (1.5, 0) {};
        \node (a3) at (3, 0)   {};
        \node (a4) at (4.5, 0) {};
        \node (a5) at (6, 0)   {};
        \node (a2u) at (1.5, 1.2) {};
        \path (a1) edge (a2);
        \path (a2) edge (a3);
        \path (a3) edge (a4);
        \path (a4) edge (a5);
        \path (a2) edge (a2u);
    \end{scope}

    \begin{scope}[xshift=8.5cm]
        \node (b1) at (0, 0)   {};
        \node (b2) at (1.5, 0) {};
        \node (b3) at (3, 0)   {};
        \node (b4) at (4.5, 0) {};
        \node (b5) at (6, 0)   {};
        \node (b3u) at (3, 1.2) {};
        \path (b1) edge (b2);
        \path (b2) edge (b3);
        \path (b3) edge (b4);
        \path (b4) edge (b5);
        \path (b3) edge (b3u);
    \end{scope}
\end{tikzpicture}
\end{center}

Nonetheless, $Z_2$ is determined by the DCP family:

\begin{cor}\label{cor:zagreb-dcp}
For a tree $T$ on $n$ vertices,
\[
Z_2(T) = \sum_{x\in V}\binom{d_x}{2}(n+1-d_x)+(n-1)-c_{3,2}(T).
\]
In particular, $Z_2(T)$ is determined by the degree sequence together with $c_{3,2}(T)$, and hence by the full family $\{\chi_d(T;k)\}_{d\geq 0}$. Consequently, two trees with the same defective chromatic polynomials have the same second Zagreb index.
\end{cor}

\begin{proof}
The formula is equation~\eqref{eq:Z2-formula}, established in the proof of Lemma~\ref{lem:P4-count}. The remaining claims follow because the degree sequence is determined by Corollary~\ref{cor:trees-degree-sequence} and $c_{3,2}(T)$ is determined by Remark~\ref{rmk:extraction}.
\end{proof}

\begin{ex}\label{ex:nine-vertices}
Consider the following two trees of order $9$:

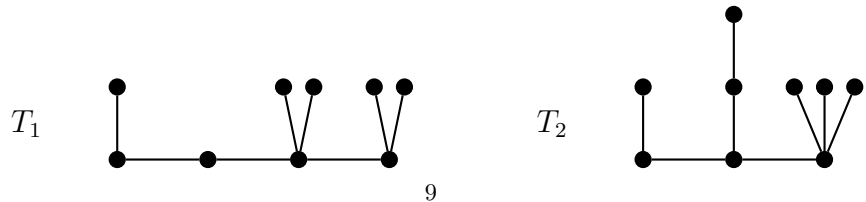
\begin{center}
\begin{tikzpicture}[
    scale=0.73,
    thick,
    every node/.style={circle,draw,fill=black,inner sep=2pt}
]
    \node (v3) at (1.5,1.2) {};
    \node (v2) at (1.5,0) {};
    \node (v1) at (3,0) {};
    \node (v0) at (4.5,0) {};
    \node (v4) at (6,0) {};
    \node (v6) at (6.25,1.2) {};
    \node (v8) at (4.25,1.2) {};
    \node (v7) at (4.75,1.2) {};
    \node (v5) at (5.75,1.2) {};

    \draw (v3)--(v2)--(v1)--(v0)--(v4)--(v6);
    \draw (v0)--(v7);
    \draw (v0)--(v8);
    \draw (v4)--(v5);

    \node[draw=none,fill=none] at (0,0.6) {$T_1$};
\end{tikzpicture}
\hspace{1.2cm}
\begin{tikzpicture}[
    scale=0.73,
    thick,
    every node/.style={circle,draw,fill=black,inner sep=2pt}
]
    \node (v8) at (1.5,1.2) {};
    \node (v7) at (1.5,0) {};
    \node (v0) at (3,0) {};
    \node (v1) at (4.5,0) {};
    \node (v2) at (4.5,1.2) {};
    \node (v5) at (3,1.2) {};
    \node (v6) at (3,2.4) {};
    \node (v4) at (4.0,1.2) {};
    \node (v3) at (5,1.2) {};

    \draw (v8)--(v7)--(v0)--(v1)--(v2);
    \draw (v0)--(v5)--(v6);
    \draw (v1)--(v3);
    \draw (v1)--(v4);

    \node[draw=none,fill=none] at (0,0.6) {$T_2$};
\end{tikzpicture}
\end{center}

These trees are not isomorphic, but they share the same degree sequence
\[
(4,3,2,2,1,1,1,1,1),
\]
and the same defective chromatic polynomials:
\begin{align*}
\chi_0(T_i;k)&=k(k-1)^8,\\
\chi_1(T_i;k)&=k^9 - 11k^7 + 20k^6 - 5k^5 - 16k^4 + 15k^3 - 4k^2,\\
\chi_2(T_i;k)&=k^9 - 5k^6 + 3k^5 + 3k^4 - 2k^3,\\
\chi_3(T_i;k)&=k^9 - k^5,\\
\chi_d(T_i;k)&=k^9\qquad (d\geq 4).
\end{align*}
Yet $N(P_5,T_1)=5$ while $N(P_5,T_2)=7$, so the full family of defective chromatic polynomials does not determine $N(P_5,T)$.
\end{ex}

\begin{rmk}
Example~\ref{ex:nine-vertices} shows that Theorem~\ref{thm:intro-path-counts} is sharp: $N(P_5,T)$ is the first path count not determined by the defective family.
\end{rmk}

\section{An infinite family}\label{sec:infinite}

Theorem~\ref{thm:intro-infinite} is a consequence of the explicit construction in Theorem~\ref{thm:infinite-family}. For each $a\geq 1$, we construct two nonisomorphic trees $X_a$ and $Y_a$ of order $a+11$ with the same family of defective chromatic polynomials. Recall that a \emph{linear forest} is a forest whose components are all paths, or equivalently, a forest of maximum degree at most $2$. The proof of Theorem~\ref{thm:infinite-family} has two parts: we first show $c_{r,\leq 2}(X_a)=c_{r,\leq 2}(Y_a)$ by counting linear forests, and then show $c_{r,\leq 1}(X_a)=c_{r,\leq 1}(Y_a)$ by proving that the trees are cospectral and using the fact that, for a tree, the characteristic polynomial equals the matching polynomial \cite[Chapter~2]{God93}.

\begin{defn}\label{def:construction}
Let $C$ be the caterpillar tree of order $11$ obtained from the path
\[
v_1\text{-}u\text{-}z_1\text{-}z_2\text{-}w\text{-}t_1\text{-}t_2\text{-}t_3\text{-}t_4
\]
by adjoining a leaf edge at $u$ and a leaf edge at $w$.

For each integer $a\geq 1$, define $X_a$ to be the tree obtained from $C$ by attaching a path with $a$ new vertices at $z_1$, and define $Y_a$ analogously by attaching a path with $a$ new vertices at $t_1$. Equivalently, $X_a$ is obtained by adding new vertices $x_1,\dots,x_a$ and edges
\[
z_1x_1,\ x_1x_2,\ \dots,\ x_{a-1}x_a,
\]
while $Y_a$ is obtained by adding new vertices $y_1,\dots,y_a$ and edges
\[
t_1y_1,\ y_1y_2,\ \dots,\ y_{a-1}y_a.
\]
\end{defn}

Both $X_a$ and $Y_a$ are trees on $a+11$ vertices with maximum degree $3$. Moreover, they are nonisomorphic: the pairwise distances among the three degree-$3$ vertices of $X_a$ form the multiset $\{1,2,3\}$, while in $Y_a$ they form $\{1,3,4\}$.

For $a=2$, this construction yields the trees $X_2$ and $Y_2$ shown below:

\medskip 

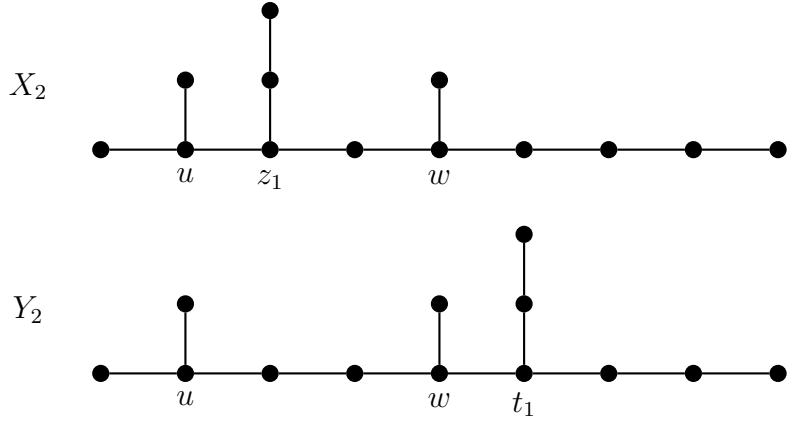
\begin{center}
\begin{tikzpicture}[scale=0.73,
    thick,
    every node/.style={circle,draw,fill=black,inner sep=2pt}]
    \begin{scope}[yshift=2.7cm, xshift=-2cm]
        \node (v1) at (0,0) {};
        \node (u)  at (1.4,0) {};
        \node (z1) at (2.8,0) {};
        \node (z2) at (4.2,0) {};
        \node (w)  at (5.6,0) {};
        \node (t1) at (7.0,0) {};
        \node (t2) at (8.4,0) {};
        \node (t3) at (9.8,0) {};
        \node (t4) at (11.2,0) {};
        \node (upu) at (1.4,1.15) {};
        \node (upw) at (5.6,1.15) {};
        \node (x1) at (2.8,1.15) {};
        \node (x2) at (2.8,2.30) {};
        \draw (v1)--(u)--(z1)--(z2)--(w)--(t1)--(t2)--(t3)--(t4);
        \draw (u)--(upu);
        \draw (w)--(upw);
        \draw (z1)--(x1)--(x2);
        \node[draw=none, fill=none, below=2pt] at (u)  {$u$};
        \node[draw=none, fill=none, below=2pt] at (z1) {$z_1$};
        \node[draw=none, fill=none, below=2pt] at (w)  {$w$};
        \node[draw=none, fill=none] at (-1.2,1.05) {$X_2$};
    \end{scope}

    \begin{scope}[yshift=-1cm, xshift=-2cm]
        \node (v1) at (0,0) {};
        \node (u)  at (1.4,0) {};
        \node (z1) at (2.8,0) {};
        \node (z2) at (4.2,0) {};
        \node (w)  at (5.6,0) {};
        \node (t1) at (7.0,0) {};
        \node (t2) at (8.4,0) {};
        \node (t3) at (9.8,0) {};
        \node (t4) at (11.2,0) {};
        \node (upu) at (1.4,1.15) {};
        \node (upw) at (5.6,1.15) {};
        \node (y1) at (7.0,1.15) {};
        \node (y2) at (7.0,2.30) {};
        \draw (v1)--(u)--(z1)--(z2)--(w)--(t1)--(t2)--(t3)--(t4);
        \draw (u)--(upu);
        \draw (w)--(upw);
        \draw (t1)--(y1)--(y2);
        \node[draw=none, fill=none, below=2pt] at (u)  {$u$};
        \node[draw=none, fill=none, below=2pt] at (w)  {$w$};
        \node[draw=none, fill=none, below=2pt] at (t1) {$t_1$};
        \node[draw=none, fill=none] at (-1.2,1.05) {$Y_2$};
    \end{scope}
\end{tikzpicture}
\end{center}

\begin{lem}\label{lem:linear-forest-counts}
For every $a\geq 1$ and every $r\geq 0$,
\[
c_{r,\leq 2}(X_a)=c_{r,\leq 2}(Y_a)
=
\binom{a+10}{r}
-3\binom{a+7}{r-3}
+\binom{a+5}{r-5}
+2\binom{a+4}{r-6}
-\binom{a+2}{r-8}.
\]
In particular, $X_a$ and $Y_a$ have the same $2$-defective chromatic polynomial.
\end{lem}

\begin{proof}
Set $m=a+10$, the common number of edges. We first count $c_{r,\leq 2}(X_a)$. Since $X_a$ has exactly three vertices of degree $3$, namely $u$, $z_1$, and $w$, an $r$-edge subset of $E(X_a)$ is a linear forest if and only if it avoids selecting all three edges incident with any one of these vertices. For each of these vertices, let $B_u$, $B_{z_1}$, and $B_w$ denote the event that all three incident edges are selected (the ``bad'' events in the inclusion--exclusion). Then
\[
c_{r,\leq 2}(X_a)=\binom{m}{r}-|B_u|-|B_{z_1}|-|B_w|
+|B_u\cap B_{z_1}|+|B_u\cap B_w|+|B_{z_1}\cap B_w|
-|B_u\cap B_{z_1}\cap B_w|.
\]
Each single bad event forces exactly three edges, so
\[
|B_u|=|B_{z_1}|=|B_w|=\binom{m-3}{r-3}.
\]
The edge sets forced by $B_u$ and $B_{z_1}$ share only the edge
$uz_1$, so their union has five edges:
\[
|B_u\cap B_{z_1}|=\binom{m-5}{r-5}.
\]
The events $B_u$ and $B_w$ force disjoint edge sets, so together they force six edges:
\[
|B_u\cap B_w|=\binom{m-6}{r-6}.
\]
Likewise, $B_{z_1}$ and $B_w$ force disjoint edge sets, so
\[
|B_{z_1}\cap B_w|=\binom{m-6}{r-6}.
\]
Finally, the triple intersection forces eight distinct edges:
\[
|B_u\cap B_{z_1}\cap B_w|=\binom{m-8}{r-8}.
\]
Substituting gives
\[
c_{r,\leq 2}(X_a)
=
\binom{m}{r}
-3\binom{m-3}{r-3}
+\binom{m-5}{r-5}
+2\binom{m-6}{r-6}
-\binom{m-8}{r-8}.
\]
Now we count $c_{r,\leq 2}(Y_a)$. Here the three degree-$3$ vertices are $u$, $w$, and $t_1$. As before, let $B_u$, $B_w$, and $B_{t_1}$ denote the corresponding bad events that all three incident edges are selected. Again, an $r$-edge subset is a linear forest if and only if it avoids all three bad events, so
\[
c_{r,\leq 2}(Y_a)=\binom{m}{r}-|B_u|-|B_w|-|B_{t_1}|
+|B_u\cap B_w|+|B_u\cap B_{t_1}|+|B_w\cap B_{t_1}|
-|B_u\cap B_w\cap B_{t_1}|.
\]
Each single bad event again forces three edges:
\[
|B_u|=|B_w|=|B_{t_1}|=\binom{m-3}{r-3}.
\]
The events $B_u$ and $B_w$ force disjoint edge sets, so
\[
|B_u\cap B_w|=\binom{m-6}{r-6}.
\]
The events $B_u$ and $B_{t_1}$ also force disjoint edge sets, so
\[
|B_u\cap B_{t_1}|=\binom{m-6}{r-6}.
\]
The edge sets forced by $B_w$ and $B_{t_1}$ share only the edge
$wt_1$, so their union has five edges:
\[
|B_w\cap B_{t_1}|=\binom{m-5}{r-5}.
\]
Finally, the triple intersection forces eight distinct edges:
\[
|B_u\cap B_w\cap B_{t_1}|=\binom{m-8}{r-8}.
\]
Thus
\[
c_{r,\leq 2}(Y_a)
=
\binom{m}{r}
-3\binom{m-3}{r-3}
+\binom{m-5}{r-5}
+2\binom{m-6}{r-6}
-\binom{m-8}{r-8}.
\]
Hence $c_{r,\leq 2}(X_a)=c_{r,\leq 2}(Y_a)$ for all $a\geq 1$ and $r\geq 0$. The final claim follows from Corollary~\ref{cor:tree-expansion}.
\end{proof}

For the second part of the proof of Theorem \ref{thm:infinite-family}, we turn to the case $d=1$. Our main tool is the following lemma of Schwenk \cite{Schwenk73}, which produces many pairs of cospectral trees; a short proof appears in \cite[Lemma~14.2.1]{BH12}.

\begin{lem}[Schwenk]\label{lem:schwenk}
Let $G$ and $G'$ be cospectral graphs and let $x\in V(G)$ and $x'\in V(G')$. Suppose that $G-x$ and $G'-x'$ are also cospectral. Let $H$ be any graph with a distinguished vertex $y$, and let $G+_{x,y}H$ denote the graph obtained by identifying $x$ with $y$. Then $G+_{x,y}H$ and $G'+_{x',y}H$ are cospectral.
\end{lem}

\begin{lem}\label{lem:cospectral-family}
For every $a\geq 1$, the trees $X_a$ and $Y_a$ are cospectral.
\end{lem}

\begin{proof}
Let $G$ and $G'$ be two copies of the caterpillar $C$ from Definition~\ref{def:construction}. Since $G\cong G'$, the graphs $G$ and $G'$ are cospectral. Let $x\in V(G)$ correspond to the vertex $z_1$, and let $x'\in V(G')$ correspond to the vertex $t_1$.

Moreover, we claim that $G-x \cong G'-x'$. Indeed, each of these forests is the disjoint union of a copy of $P_3$ and the $7$-vertex tree obtained from $P_6$ by attaching a leaf at the vertex adjacent to an endpoint. In particular, $G-x$ and $G'-x'$ are cospectral.

Let $H_a=P_{a+1}$, and distinguish one endpoint $y$ of $H_a$. Then $G+_{x,y}H_a$ is exactly $X_a$, while $G'+_{x',y}H_a$ is exactly $Y_a$. By Lemma~\ref{lem:schwenk}, the two trees are cospectral.
\end{proof}

\begin{cor}\label{cor:matching-family}
For every $a\geq 1$ and every $r\geq 0$, we have
\[
c_{r,\leq 1}(X_a)=c_{r,\leq 1}(Y_a).
\]
Consequently, $X_a$ and $Y_a$ have the same $1$-defective chromatic polynomial.
\end{cor}

\begin{proof}
For a tree $T$, the coefficient $c_{r,\leq 1}(T)$ counts the $r$-edge matchings of $T$; equivalently, these numbers are the coefficients of the matching polynomial of $T$. Since $X_a$ and $Y_a$ are trees, their characteristic polynomials agree with their matching polynomials \cite[Chapter~2]{God93}. Lemma~\ref{lem:cospectral-family} therefore implies that the matching polynomials of $X_a$ and $Y_a$ are equal, and hence so are the numbers $c_{r,\leq 1}(X_a)$ and $c_{r,\leq 1}(Y_a)$ for all $r$.

The final statement now follows from Corollary~\ref{cor:tree-expansion}.
\end{proof}

\begin{thm}\label{thm:infinite-family}
For every $a\geq 1$, the trees $X_a$ and $Y_a$ are nonisomorphic and have the same full family of defective chromatic polynomials. In particular, for every integer $n\geq 12$, there exists such a pair on $n$ vertices.
\end{thm}

\begin{proof}
The trees are nonisomorphic by the discussion following Definition~\ref{def:construction}.

By Lemma~\ref{lem:linear-forest-counts}, the two trees have the same $2$-defective chromatic polynomial. By Corollary~\ref{cor:matching-family}, they have the same $1$-defective chromatic polynomial. Since both are trees on $a+11$ vertices, they also have the same ordinary chromatic polynomial,
\[
\chi_0(X_a;k)=\chi_0(Y_a;k)=k(k-1)^{a+10}.
\]
Finally, both trees have maximum degree $3$, so
\[
\chi_d(X_a;k)=\chi_d(Y_a;k)=k^{a+11}\qquad\text{for all }d\geq 3.
\]
Thus, $\chi_d(X_a;k)=\chi_d(Y_a;k)$ for every $d\geq 0$. Taking $a=n-11$ yields the final statement.
\end{proof}

\begin{rmk}
The trees in the family in Theorem~\ref{thm:infinite-family} are not caterpillars unless $a=1$. However, one can modify the construction to obtain a caterpillar analogue. For each integer $b\geq 1$, let $\widetilde{X}_b$ be the tree obtained from the base caterpillar $C$ in Definition~\ref{def:construction} by identifying $z_1$ with the center of a copy of $K_{1,b}$, and let $\widetilde{Y}_b$ be obtained analogously by identifying $t_1$ with the center of a copy of $K_{1,b}$. Equivalently, $\widetilde{X}_b$ is obtained from $C$ by adjoining $b$ new leaves at $z_1$, while $\widetilde{Y}_b$ is obtained from $C$ by adjoining $b$ new leaves at $t_1$. Thus $\widetilde{X}_b$ and $\widetilde{Y}_b$ are caterpillar trees on $b+11$ vertices. They are nonisomorphic, since the pairwise distance multiset among their vertices of degree at least $3$ is $\{1,2,3\}$ for $\widetilde{X}_b$ and $\{1,3,4\}$ for $\widetilde{Y}_b$.

For illustration, when $b=3$ we obtain the following pair:

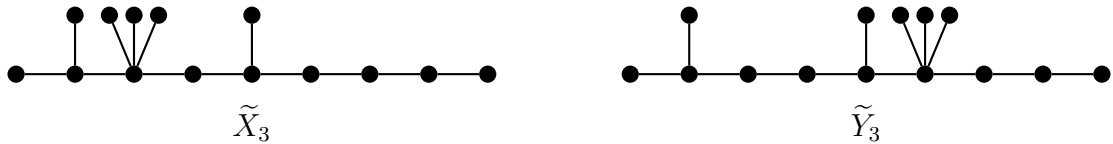
\begin{center}
\begin{tikzpicture}[
    scale=0.68,
    thick,
    every node/.style={circle,draw,fill=black,inner sep=2pt}
]
    \node (xv1) at (0,0) {};
    \node (xu)  at (1.2,0) {};
    \node (xz1) at (2.4,0) {};
    \node (xz2) at (3.6,0) {};
    \node (xw)  at (4.8,0) {};
    \node (xt1) at (6.0,0) {};
    \node (xt2) at (7.2,0) {};
    \node (xt3) at (8.4,0) {};
    \node (xt4) at (9.6,0) {};
    \node (xuL) at (1.2,1.2) {};
    \node (xwL) at (4.8,1.2) {};
    \node (xzU) at (1.9,1.2) {};
    \node (xzD) at (2.9,1.2) {};
    \node (xzS) at (2.4,1.2) {};

    \draw (xv1)--(xu)--(xz1)--(xz2)--(xw)--(xt1)--(xt2)--(xt3)--(xt4);
    \draw (xu)--(xuL);
    \draw (xw)--(xwL);
    \draw (xz1)--(xzU);
    \draw (xz1)--(xzD);
    \draw (xz1)--(xzS);

    \node[draw=none,fill=none] at (4.8,-1) {$\widetilde{X}_3$};
    \node (yv1) at (12.5,0) {};
    \node (yu)  at (13.7,0) {};
    \node (yz1) at (14.9,0) {};
    \node (yz2) at (16.1,0) {};
    \node (yw)  at (17.3,0) {};
    \node (yt1) at (18.5,0) {};
    \node (yt2) at (19.7,0) {};
    \node (yt3) at (20.9,0) {};
    \node (yt4) at (22.1,0) {};
    \node (yuL) at (13.7,1.2) {};
    \node (ywL) at (17.3,1.2) {};
    \node (ytU) at (18,1.2) {};
    \node (ytD) at (19,1.2) {};
    \node (ytS) at (18.5,1.2) {};

    \draw (yv1)--(yu)--(yz1)--(yz2)--(yw)--(yt1)--(yt2)--(yt3)--(yt4);
    \draw (yu)--(yuL);
    \draw (yw)--(ywL);
    \draw (yt1)--(ytU);
    \draw (yt1)--(ytD);
    \draw (yt1)--(ytS);

    \node[draw=none,fill=none] at (17.3,-1) {$\widetilde{Y}_3$};
\end{tikzpicture}
\end{center}

The proof that $\widetilde{X}_b$ and $\widetilde{Y}_b$ have the same full family of defective chromatic polynomials is parallel to the proof of Theorem \ref{thm:infinite-family} above. Applying Lemma~\ref{lem:schwenk} with $H=K_{1,b}$, taking the distinguished vertex to be the center, shows $\widetilde{X}_b$ and $\widetilde{Y}_b$ are cospectral, so the matching-polynomial argument of Corollary~\ref{cor:matching-family} yields equality of the $1$-defective chromatic polynomials. The case $d=2$ is an inclusion--exclusion count as in
Lemma~\ref{lem:linear-forest-counts}. In both trees, the attachment vertex has degree $b+2$, is adjacent to exactly one of the two degree-$3$ vertices, and shares no incident edge with the other. Thus the same inclusion--exclusion count applies to $\widetilde{X}_b$ and $\widetilde{Y}_b$. For $d\geq 3$, the degree-$3$ vertices never create a violation of $\Delta(\widetilde{X}_b[A])\leq d$. Hence the only possible obstruction is at the attachment vertex. Counting according to the number $s$ of chosen edges incident with this
vertex gives
\[
c_{r,\leq d}(\widetilde{X}_b)
=
c_{r,\leq d}(\widetilde{Y}_b)
=
\sum_{s=0}^{\min\{d,b+2\}}
\binom{b+2}{s}\binom{8}{r-s}.
\]
Combining these with the case $d=0$, we conclude that $\chi_d(\widetilde{X}_b;k)=\chi_d(\widetilde{Y}_b;k)$ for every $d\geq 0$. In particular, for every $n\geq 12$, taking $b=n-11$ yields a nonisomorphic pair of caterpillars on $n$ vertices with the same full family of defective chromatic polynomials.
\end{rmk}

\begin{rmk} While Theorem~\ref{thm:infinite-family} yields pairs of nonisomorphic trees on $n$ vertices with the same DCP family for every $n\geq 12$, the smaller orders $n=9, 10, 11$ also admit such pairs. By computer search, no such pair exists for $n\leq 8$. For $n=9$, the pair from Example~\ref{ex:nine-vertices} is the unique such pair. For $n=10$, there is also a unique such pair, shown below:

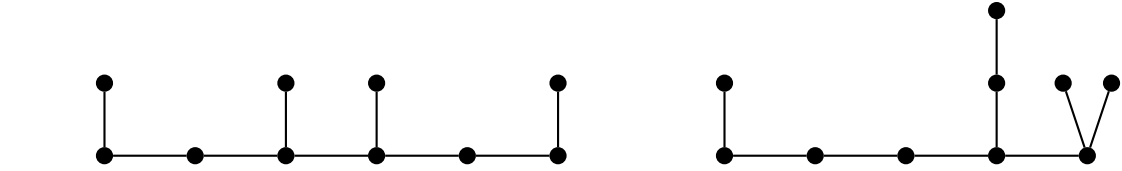
\begin{center}
\begin{tikzpicture}[
    scale=0.73,
    thick,
    every node/.style={
        circle,
        draw,
        fill=black,
        inner sep=2pt
    }
]
    \node (v8) at (1.5,1.2) {};
    \node (v7) at (1.5,0) {};
    \node (v6) at (3,0) {};
    \node (v0) at (4.5,0) {};
    \node (v1) at (6,0) {};
    \node (v2) at (7.5,0) {};
    \node (v3) at (9,0) {};
    \node (v4) at (9,1.2) {};
    \node (v9) at (4.5,1.2) {};
    \node (v5) at (6,1.2) {};

    \draw (v8) -- (v7) -- (v6) -- (v0) -- (v1) -- (v2) -- (v3) -- (v4);
    \draw (v0) -- (v9);
    \draw (v1) -- (v5);
\end{tikzpicture}
\hspace{1.7cm}
\begin{tikzpicture}[
    scale=0.73,
    thick,
    every node/.style={
        circle,
        draw,
        fill=black,
        inner sep=2pt
    }
]
    \node (v9) at (1.5,1.2) {};
    \node (v8) at (1.5,0) {};
    \node (v7) at (3,0) {};
    \node (v0) at (4.5,0) {};
    \node (v1) at (6,0) {};
    \node (v2) at (7.5,0) {};
    \node (v3) at (7.9,1.2) {};
    \node (v5) at (6,1.2) {};
    \node (v6) at (6,2.4) {};
    \node (v4) at (7.1,1.2) {};

    \draw (v9) -- (v8) -- (v7) -- (v0) -- (v1) -- (v2) -- (v3);
    \draw (v1) -- (v5) -- (v6);
    \draw (v2) -- (v4);
\end{tikzpicture}
\end{center}

For $n=11$, there are exactly six such pairs; one of them is shown below:

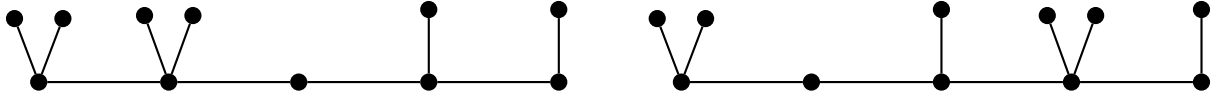
\begin{center}
\begin{tikzpicture}[scale=0.73,
    thick,
    every node/.style={circle,draw,fill=black,inner sep=2pt}]
    \node (v2) at (0,0) {};
    \node (v1) at (2.15,0) {};
    \node (v0) at (4.30,0) {};
    \node (v7) at (6.45,0) {};
    \node (v8) at (8.60,0) {};
    \node (v3)  at (-0.40,1.05) {};
    \node (v4)  at (0.40,1.05) {};
    \node (v5)  at (1.75,1.10) {};
    \node (v6)  at (2.55,1.10) {};
    \node (v10) at (6.45,1.20) {};
    \node (v9)  at (8.60,1.20) {};
    \draw (v2)--(v1)--(v0)--(v7)--(v8);
    \draw (v2)--(v3);
    \draw (v2)--(v4);
    \draw (v1)--(v5);
    \draw (v1)--(v6);
    \draw (v7)--(v10);
    \draw (v8)--(v9);
\end{tikzpicture}
\hspace{0.8cm}
\begin{tikzpicture}[scale=0.73,
    thick,
    every node/.style={circle,draw,fill=black,inner sep=2pt}]
    \node (v2) at (0,0) {};
    \node (v1) at (2.15,0) {};
    \node (v0) at (4.30,0) {};
    \node (v5) at (6.45,0) {};
    \node (v6) at (8.60,0) {};
    \node (v3)  at (-0.40,1.05) {};
    \node (v4)  at (0.40,1.05) {};
    \node (v10) at (4.30,1.20) {};
    \node (v8)  at (6.05,1.10) {};
    \node (v9)  at (6.85,1.10) {};
    \node (v7)  at (8.60,1.20) {};
    \draw (v2)--(v1)--(v0)--(v5)--(v6);
    \draw (v2)--(v3);
    \draw (v2)--(v4);
    \draw (v0)--(v10);
    \draw (v5)--(v8);
    \draw (v5)--(v9);
    \draw (v6)--(v7);
\end{tikzpicture}
\end{center}
\end{rmk}

\section{Comparison between DCP and other graph invariants}\label{sec:comparison}

In this section, we compare the defective chromatic polynomials for trees with several other graph invariants. Consider the pair of trees of order $9$ from Example~\ref{ex:nine-vertices}. 

\begin{center}
\begin{tikzpicture}[
    scale=0.73,
    thick,
    every node/.style={circle,draw,fill=black,inner sep=2pt}
]
    \node (v3) at (1.5,1.2) {};
    \node (v2) at (1.5,0) {};
    \node (v1) at (3,0) {};
    \node (v0) at (4.5,0) {};
    \node (v4) at (6,0) {};
    \node (v6) at (6.25,1.2) {};
    \node (v8) at (4.25,1.2) {};
    \node (v7) at (4.75,1.2) {};
    \node (v5) at (5.75,1.2) {};

    \draw (v3)--(v2)--(v1)--(v0)--(v4)--(v6);
    \draw (v0)--(v7);
    \draw (v0)--(v8);
    \draw (v4)--(v5);

    \node[draw=none,fill=none] at (0,0.6) {$T_1$};
\end{tikzpicture}
\hspace{1.2cm}
\begin{tikzpicture}[
    scale=0.73,
    thick,
    every node/.style={circle,draw,fill=black,inner sep=2pt}
]
    \node (v8) at (1.5,1.2) {};
    \node (v7) at (1.5,0) {};
    \node (v0) at (3,0) {};
    \node (v1) at (4.5,0) {};
    \node (v2) at (4.5,1.2) {};
    \node (v5) at (3,1.2) {};
    \node (v6) at (3,2.4) {};
    \node (v4) at (4.0,1.2) {};
    \node (v3) at (5,1.2) {};

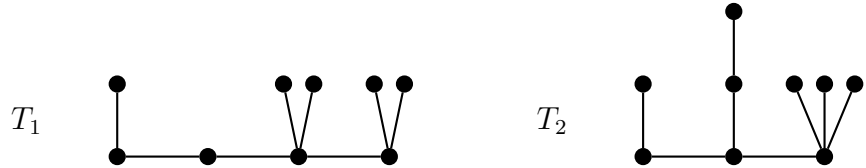
    \draw (v8)--(v7)--(v0)--(v1)--(v2);
    \draw (v0)--(v5)--(v6);
    \draw (v1)--(v3);
    \draw (v1)--(v4);

    \node[draw=none,fill=none] at (0,0.6) {$T_2$};
\end{tikzpicture}
\end{center}

The two trees $T_1$ and $T_2$ share the same full family of defective chromatic polynomials. We now exhibit several graph invariants that distinguish $T_1$ from $T_2$, showing that these invariants are not determined by the DCP family. All numerical data below were computed in SageMath.

First, $T_1$ has diameter $5$, while $T_2$ has diameter $4$. Similarly, $T_1$ has radius $3$, while $T_2$ has radius $2$. Hence, DCP does not determine these extremal distance-based invariants.

Recall that the \emph{independence polynomial} of a graph $G$ is the generating function
\[
I(G,x)=\sum_{j\ge 0} s_j x^j,
\]
where $s_j$ denotes the number of independent sets of $G$ of size $j$. We have 
\begin{align*}
I(T_1, x)&=x^6 + 8x^5 + 26x^4 + 39x^3 + 28x^2 + 9x + 1 \\
I(T_2, x)&=x^6 + 9x^5 + 26x^4 + 39x^3 + 28x^2 + 9x + 1
\end{align*}
Note that $T_1$ and $T_2$ are cospectral, since they have the same $1$-defective chromatic polynomial; we therefore turn to a different spectral invariant. The \emph{Laplacian matrix} of a graph $T$ is $L(T)=D(T)-A(T)$, where $D(T)$ is the diagonal degree matrix and $A(T)$ is the adjacency matrix. The \emph{Laplacian polynomial} of $T$ is the characteristic polynomial of $L(T)$:
\[
\phi_L(T,x)=\det(xI-L(T)).
\]
The two trees $T_1$ and $T_2$ have different Laplacian polynomials:
\begin{align*}
\phi_{L}(T_1, x) &= x^9 - 16x^8 + 101x^7 - 326x^6 + 584x^5 - 592x^4 + 329x^3 - 90x^2 + 9x, \\
\phi_{L}(T_2, x) &= x^9 - 16x^8 + 101x^7 - 326x^6 + 584x^5 - 590x^4 + 325x^3 - 88x^2 + 9x.
\end{align*}
Following Dohmen, P\"onitz, and Tittmann \cite{DPT03}, the \emph{DPT polynomial} $P(G;x,y)$ counts vertex colorings with $x$ colors in which the first $y$ colors are proper and the remaining $x-y$ colors may be repeated on adjacent vertices:
\[
P(G;x,y)=\sum_{X\subseteq V(G)} (x-y)^{|X|}\,\chi(G-X;y).
\]
In particular, $P(G;x,x)$ is the ordinary chromatic polynomial $\chi_0(G; x)$. 

The following pair $T_3$ and $T_4$ shows that the DPT polynomial does not determine the defective chromatic polynomials of a tree. 

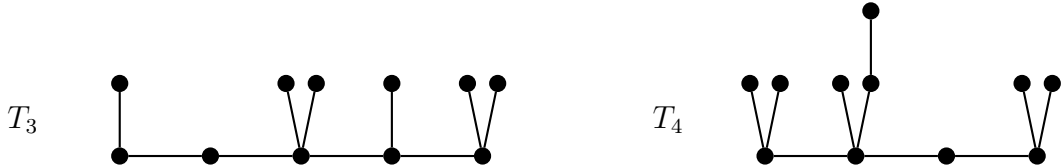
\begin{center}
\begin{tikzpicture}[
    scale=0.73,
    thick,
    every node/.style={circle,draw,fill=black,inner sep=2pt}
]
    \node (v8) at (1.5,1.2) {};
    \node (v7) at (1.5,0) {};
    \node (v6) at (3,0) {};
    \node (v0) at (4.5,0) {};
    \node (v1) at (6,0) {};
    \node (v2) at (7.5,0) {};
    \node (v3) at (7.75,1.2) {};
    \node (v9)  at (4.25,1.2) {};
    \node (v10) at (4.75,1.2) {};
    \node (v5)  at (6,1.2) {};
    \node (v4)  at (7.25,1.2) {};
    \draw (v8)--(v7)--(v6)--(v0)--(v1)--(v2)--(v3);
    \draw (v0)--(v9);
    \draw (v0)--(v10);
    \draw (v1)--(v5);
    \draw (v2)--(v4);
    \node[draw=none,fill=none] at (-0.1,0.6) {$T_3$};
\end{tikzpicture}
\hspace{1.5cm}
\begin{tikzpicture}[
    scale=0.73,
    thick,
    every node/.style={circle,draw,fill=black,inner sep=2pt}
]
    \node (v7) at (1.25,1.2) {};
    \node (v5) at (1.5,0) {};
    \node (v0) at (3,0) {};
    \node (v1) at (4.5,0) {};
    \node (v2) at (6,0) {};
    \node (v3) at (6.25,1.2) {};
    \node (v6) at (1.75,1.2) {};
    \node (v10) at (2.75,1.2) {};
    \node (v8) at (3.25,1.2) {};
    \node (v9) at (3.25,2.4) {};
    \node (v4) at (5.75,1.2) {};
    \draw (v7)--(v5)--(v0)--(v1)--(v2)--(v3);
    \draw (v5)--(v6);
    \draw (v0)--(v10);
    \draw (v0)--(v8)--(v9);
    \draw (v2)--(v4);
    \node[draw=none,fill=none] at (-0.1,0.6) {$T_4$};
\end{tikzpicture}
\end{center}

Indeed, $T_3$ and $T_4$ have the same DPT polynomial, but different $2$-defective chromatic polynomials:
\begin{align*}
\chi_2(T_3;k)&=k^{11}-6k^8+3k^7+4k^6+2k^5-6k^4+2k^3, \\ 
\chi_2(T_4;k) &=k^{11}-6k^8+3k^7+3k^6+3k^5-3k^4-3k^3+2k^2.
\end{align*}
Next, we consider another polynomial invariant of a tree $T=(V, E)$. For a vertex subset $A\subseteq V$, define
\[
E(A)=\{\text{edges of }E\text{ with both endpoints in }A\}, \qquad e(A)=|E(A)|,
\]
\[
D(A)=\{\text{edges of }E\text{ with exactly one endpoint in }A\}, \qquad d(A)=|D(A)|.
\]
Also define
\[
g_T(a,b,c)=\bigl|\{A\subseteq V(T): |A|=a,\ d(A)=b,\ e(A)=c\}\bigr|.
\]
The \emph{generalized degree polynomial} of $T$ is
\[
G_T=G_T(x,y,z)=\sum_{A\subseteq V} x^{|A|}y^{d(A)}z^{e(A)}
= \sum_{a,b,c} g_T(a,b,c)x^a y^b z^c.
\]
The GDP was introduced by Crew \cite{Cre20}, and further studied in \cite{Cre22}. Aliste-Prieto, Martin, Wagner, and Zamora \cite{A-PMWZ24} proved Crew's conjecture \cite{Cre22}: the chromatic symmetric function (CSF) of a tree $T$ determines the polynomial $G_T(x, y, z)$. On the other hand, Liu and Tang \cite{LT26} showed that the GDP of a tree determines its double-degree sequence, leaf-adjacency sequence, and the component-size multiset of $T_{(2)}$, where $T_{(2)}$ is the induced subgraph on the degree-2 vertices.

The two $9$-vertex trees $T_1$ and $T_2$ have the same DCP family but different GDPs. Conversely, the trees $T_5$ and $T_6$ below have the same GDP but different $2$-defective chromatic polynomials, and therefore different DCP families.

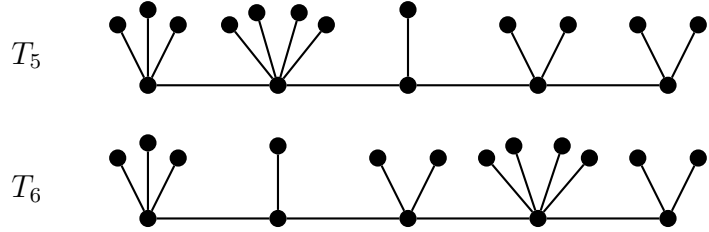
\begin{center}
\begin{tikzpicture}[scale=0.73,
    thick,
    every node/.style={circle,draw,fill=black,inner sep=2pt}]
    \node (v2)  at (0,0) {};
    \node (v1)  at (2.15,0) {};
    \node (v0)  at (4.30,0) {};
    \node (v10) at (6.45,0) {};
    \node (v11) at (8.60,0) {};
    \node (v3) at (-0.50,1.00) {};
    \node (v4) at (0,1.25) {};
    \node (v5) at (0.50,1.00) {};
    \node (v6) at (1.35,1.00) {};
    \node (v7) at (1.80,1.20) {};
    \node (v8) at (2.50,1.20) {};
    \node (v9) at (2.95,1.00) {};
    \node (v16) at (4.30,1.25) {};
    \node (v14) at (5.95,1.00) {};
    \node (v15) at (6.95,1.00) {};
    \node (v12) at (8.10,1.00) {};
    \node (v13) at (9.10,1.00) {};
    \draw (v2)--(v1)--(v0)--(v10)--(v11);
    \draw (v2)--(v3);
    \draw (v2)--(v4);
    \draw (v2)--(v5);
    \draw (v1)--(v6);
    \draw (v1)--(v7);
    \draw (v1)--(v8);
    \draw (v1)--(v9);
    \draw (v0)--(v16);
    \draw (v10)--(v14);
    \draw (v10)--(v15);
    \draw (v11)--(v12);
    \draw (v11)--(v13);
    \node[draw=none,fill=none] at (-2,0.5) {$T_5$};
\end{tikzpicture}

\vspace{0.5cm}

\begin{tikzpicture}[scale=0.73,
    thick,
    every node/.style={circle,draw,fill=black,inner sep=2pt}]
    \node (v2) at (0,0) {};
    \node (v1) at (2.15,0) {};
    \node (v0) at (4.30,0) {};
    \node (v7) at (6.45,0) {};
    \node (v8) at (8.60,0) {};
    \node (v3) at (-0.50,1.00) {};
    \node (v4) at (0,1.25) {};
    \node (v5) at (0.50,1.00) {};
    \node (v6) at (2.15,1.20) {};
    \node (v15) at (3.80,1.00) {};
    \node (v16) at (4.80,1.00) {};
    \node (v11) at (5.60,1.00) {};
    \node (v12) at (6.05,1.20) {};
    \node (v13) at (6.85,1.20) {};
    \node (v14) at (7.30,1.00) {};
    \node (v9)  at (8.10,1.00) {};
    \node (v10) at (9.10,1.00) {};
    \draw (v2)--(v1)--(v0)--(v7)--(v8);
    \draw (v2)--(v3);
    \draw (v2)--(v4);
    \draw (v2)--(v5);
    \draw (v1)--(v6);
    \draw (v0)--(v15);
    \draw (v0)--(v16);
    \draw (v7)--(v11);
    \draw (v7)--(v12);
    \draw (v7)--(v13);
    \draw (v7)--(v14);
    \draw (v8)--(v9);
    \draw (v8)--(v10);
    \node[draw=none,fill=none] at (-2,0.5) {$T_6$};
\end{tikzpicture}
\end{center}

Taken together, the pairs $(T_1,T_2)$, $(T_3,T_4)$, and $(T_5,T_6)$ show that, as tree invariants, DCP is incomparable with each of DPT and GDP.

\begin{table}[ht]
\centering
\begin{tabular}{lccc}
\toprule
Invariant & Same $(T_1,T_2)$? & Same $(T_3,T_4)$? & Same $(T_5,T_6)$? \\
\midrule
DCP family & Yes & No & No \\
DPT polynomial & No & Yes & Yes \\
GDP & No & No & Yes \\
Adjacency spectrum & Yes & Yes & Yes \\
Laplacian spectrum & No & No & Yes \\
Independence polynomial & No & Yes & Yes \\
\bottomrule \\ 
\end{tabular}
\caption{Comparison of invariants on three pairs of nonisomorphic trees. A ``Yes'' entry means the invariant agrees on the pair; a ``No'' entry means it distinguishes them.}
\label{tab:invariant-comparison}
\end{table}

We conclude with an open question motivated by the known implication CSF $\Rightarrow$ GDP \cite[Theorem~6]{A-PMWZ24}.

\begin{quest}
Suppose $T_1$ and $T_2$ are two trees with the same chromatic symmetric function. Is it true that $\chi_{d}(T_1;k)=\chi_{d}(T_2;k)$ for all $d\geq 0$?
\end{quest}

\section*{Acknowledgments}
We thank the anonymous referees for their careful reading and helpful suggestions, which improved the presentation of this paper.

\bibliographystyle{alpha}
\bibliography{main}

\end{document}